\documentclass[11pt]{article}
\usepackage{amsmath,amsthm,amssymb,ifsym}
\usepackage[twoside]{geometry}
\usepackage{graphicx}
\usepackage{epsfig}
\usepackage{mathpazo} % math & rm
\usepackage[scaled]{helvet} % ss
\usepackage{courier} % tt
\normalfont
\usepackage[T1]{fontenc}

\newtheorem{theorem}{Theorem}[section]
\newtheorem{proposition}[theorem]{Proposition}
\newtheorem{lemma}[theorem]{Lemma}

\numberwithin{equation}{section}
\numberwithin{figure}{section}
\usepackage{amssymb, amscd, mathrsfs, wasysym}

\newcommand \tildeU {\widetilde U}

\newcommand \trianglerightNEW \triangleright

\newcommand \sgn {\text{sgn}}

\newcommand \auth {\textsc}

\newcommand \bei {\begin{itemize}}
\newcommand \eei {\end{itemize}}
\newcommand \be {\begin{equation}}
\newcommand \bel {\begin{equation}\label}
\newcommand \ee {\end{equation}}
\newcommand \del \partial

\newcommand \eps \epsilon
\let\oldmarginpar\marginpar
\renewcommand\marginpar[1]{\-\oldmarginpar[\raggedleft\footnotesize #1]%
{\raggedright\footnotesize #1}}
\begin{document}

\title{\bf \Large Existence theory for well-balanced Euler model
}
\author{Shuyang XIANG \footnote{ \normalsize MATHCCES, Department of Mathematics, RWTH Aachen University, 
Schinkelstrasse 2, D-52062 Aachen, Germany.
 E-mail : {\sl xiang@mathcces.rwth-aachen.de}}, 
Yangyang CAO \footnote{
\normalsize Laboratoire Jacques-Louis Lions \& Centre National de la Recherche Scientifique,
Sorbonne Universit\'e, 4 Place Jussieu, 75252 Paris, France.
E-mail : {\sl caoy@ljll.math.upmc.fr}}
\,
}
\maketitle

\abstract{
We study the initial value problem for  a kind of Euler equation with a source term. Our main result is  the existence of a globally-in-time weak solution  whose total variation is bounded on the the domain of definition,  allowing the existence of shock waves. Our proof relies on a well-balanced random choice
method called Glimm method which preserves the fluid equilibria  and we construct a sequence of approximate weak solutions which converges to the exact weak solution of the initial value problem, based on the construction of exact solutions of the generalized Riemann problem associated with initially piecewise steady state solutions.}

\setcounter{tocdepth}{4}

%======================================================
\section{Introduction}
Our model of interest   is a non-conservative Euler equation with a source term reading 

\bel{Euler1}
\aligned
&\del_t\rho  +\del_r(\rho v) +{2 \over r } \rho v =0  ,
\\
&\del_t(\rho v) +\del_r\Big(\rho( v^2+k^2) \Big) +{2\over r}  \rho v^2+{1\over r^2 } m\rho= 0,
\endaligned
\ee
defined for all $r>0$ where the main unknowns are the density $\rho>0$ and the velocity $v $ of a fluid flow in consideration. The model \eqref{Euler1} is indeed the "non-relativistic version" of the  Euler equation on a Schwarzschild spacetime background studied by LeFloch and Xiang~\cite{PLF-SX-one} where a well-posedness theory was given  for the relativist model. Here, the parameters are given as the Schwarzschild black hole mass $m \in (0, +\infty)$ and the constant sound speed $k \in (0,+\infty)$.  An interesting observation is that remark that even if the Euler model \eqref{Euler1} is non-relativistic in the sense that the velocity $v$ is far from light speed, the mass of the black hole $m$ is still reflected by the source term. 

Our model has the  form of a well-balanced hyperbolic system with the right-hand side source terms because of the  geometry of the Schwarszhchild space. Such well-balanced system was first investigated by Dafermos and Hsiao~\cite{DL}, Liu~\cite{Liu2}, for different applications. In our investigation, we closely follow  LeFloch and Xiang~\cite{PLF-SX-one}, which treated the relativistic version of the Euler model by allowing the fluid speed comparable to the speed of light. However, in our non-relativistic case, we were able to get rid of the influence of the light speed and had some stronger results. 

Our main contributions of the Euler model with a source terms \eqref{Euler1} are listed as follows:
\bei
\item A systematic study of the  existence the steady state solutions. 
\item The global-in-time existence of the (triple) generalized Riemann problem, which is an initial problem of \eqref{Euler1} with a given piecewise steady state. Moreover, we gave also an analytical formulation of the exact solution. 
\item The existence of the Euler model \eqref{Euler1} with an arbitrary initial data with bounded total variation. 
\eei

The organization of this paper is as follows. In Section~\ref{sec:2} we give some basic properties of the homogenous Euler model without source term, including the hyperbolicity and the nonlinear properties which lead us to give the result of the standard Riemann problem whose wave interactions are analyzed as well.

We take into consideration the steady state solutions in Section~\ref{sec:3}, where we first study  different families of smooth steady state solutions to the Euler model, serving as one of the main results of the present paper. The study coming after is the generalized Riemann problem  of the Euler model with the  initial data consisting  of  two steady state solutions separated by a discontinuity of jump.  An exact solution is constructed  in Section~\ref{sec:4}, with three steady states connected by two different families of generalized elementary waves and we have verified that the Rankie-Hugoniot jump condition and the Lax entropy condition are satisfied. We also give the evolution of the total variation of the solution of the Riemann problem. 

Referring to Section~\ref{sec:3}, smooth steady states may not  be extended on the whole space region $(0, +\infty)$. To give a complete construction of an initial value problem, it is necessary to consider a so-called  triple Riemann problem, which is an initial problem with its initial data given as three steady state solutions separated by two given radius. Such problem was first studied by Lefloch and Xiang~\cite{PLF-SX-two} for a Burgers model on the Schwarzschild spacetime. We provide a global-in-time solution of such problem for our model in Section~\ref{sec:5}.

In Section~\ref{sec:6}, we are then able to give  an existence theory of our Euler  model. Inspired by the classic Glimm method~\cite{GLIMM} and the application of such method in the case of fluid flows in a flat space~\cite{Nishida, SJ}, we generalize the method based on the (triple) generalized Riemann problem, developed earlier in~\cite{GL,PLF-SX-two} in a different geometric setup and provides us with the desired global-in-time result.  For the fluids of the Euler model in consideration in the present  paper, the geometry may leads to the growth of the total variation of the solution, but we prove that it is  uniformly controlled on any compact interval of time and consequently,  sequence is proved to converge to  the exact  global-in-time solution  of  the Euler model \eqref{Euler1}.

 %=================
\section{Homogenous system  }
\label{sec:2}
\subsection{Elementary waves}
According to \eqref{Euler1}, we write the Euler system as
\bel{Euler-form}
\del_ t U+ \del_ r F (U )= S (r, U),
\ee
where
\[
U =\bigg(
\begin{array}{cccc}
 \rho\\
\rho  v \\
\end{array}
\bigg), \qquad
F (U )= \bigg(
\begin{array}{cccc}
 \rho v \\
 \rho( v^2+k^2) \\
\end{array}\bigg), \qquad
S (r, U)=   \bigg(
\begin{array}{cccc}
- {2 \over r } \rho v  \\
-{2\over r}  \rho v^2- {1\over r^2 } m\rho
\end{array}\bigg).
\]
We derive the pair of eigenvalues reading
\bel{eigenvalue}
\lambda(\rho, v)= v- k,  \qquad \mu(\rho, v) = v + k.
\ee
We give also the pair of corresponding Riemann invariants:
\bel{Rie-in}
w(\rho, v)= v+ k \ln \rho, \qquad z(\rho, v) = v -  k   \ln\rho.
\ee

Following directly from \eqref{eigenvalue}, we have the following proposition:
\begin{proposition}
\label{hyperbolicity}
Let $k>0$ be the sound speed and $m>0$ the black hole mass,  the  non-conservative Euler model \eqref{Euler1} is strictly hyperbolic and both  characteristic fields are genuinely nonlinear.
\end{proposition}

Proposition~\ref{hyperbolicity} enables us to  consider first  the elementary waves of the homogenous Euler system:
\bel{homo}
\del_ t U+ \del_ r F (U )=0,
\ee
where we recall that $U= (\rho, \rho  v ) ^{T}$ and $
F (U )=\big (
 \rho v ,
 \rho( v^2+k^2) \big)^{T}$ according to \eqref{Euler-form}.  Notice that $(\rho, v)\to (\rho, \rho  v ) $ is a one-to-one map and we thus don't distinguish $U$ and $(\rho, v)$ in the following for the sake of simplicity.

We consider first the  rarefaction curves along which the corresponding Riemann invariants  remain constant.

\begin{lemma}
Consider the homogenous Euler model given by \eqref{homo}. The 1-rarefaction curve issuing from constant $U_L = (\rho_L, v_L)$ and the 2-rarefaction wave from the constant $U_R= (\rho_R, v_R)$ are given  by
\bel{lemma-rarefaction}
R_1^\rightarrow  (U_L ): \bigg\{ v - v_L = \ln \Big({\rho\over \rho_L} \Big)^ {-k},   \quad v < v_L \bigg\}, \quad R_2^\leftarrow  (U_R ):  \bigg\{ v - v_R=  \ln \Big({\rho\over \rho_R} \Big)^ k,  \quad v < v_R \bigg\}.
\ee
\end{lemma}
\begin{proof}
The 1-family Riemann invariant is a constant along the 1-rarefaction curve   passing  the point $U_L$ and we have
 \[
 R_1^\rightarrow (U_L): w(\rho, v) = w(\rho_L, v_L),  \quad z(\rho, v) < z(\rho_L, v_L),
 \]
which gives the form of the 1-rarefaction wave. Similarly, we have the 2-rarefaction wave.
\end{proof}

We can also give the form of 1-shock and 2-shock associated with the constant states $U_L$ and $ U_R$ respectively.
\begin{lemma}
The 1-shock wave and 2-shock wave of the Euler model without source term \eqref{homo} associated with the constant states $U_L$ and $ U_R$ respectively have the following forms:
\bel{lemma-shock}
\aligned
& S_1^\rightarrow(U_L):  \bigg\{ v - v_L = -k \Big(\sqrt {\rho\over \rho_L} - \sqrt {\rho_L \over \rho}\Big), \quad  v > v_L \bigg\}, \\
& S_2^\leftarrow(U_R):  \bigg\{ v - v_R = k \Big(\sqrt {\rho\over \rho_R} - \sqrt {\rho_R \over \rho}\Big), \quad  v > v_R \bigg\}.
\endaligned
\ee
And the 1-shock speed $\sigma_1$  and the 2-speed $\sigma_2$  are:
\bel{shock-speed}
\sigma_1 \big((\rho_L, v_L), (\rho, v)\big)= v - k\sqrt{{\rho_L \over\rho}}, \quad
 \sigma_2\big((\rho, v), (\rho_R, v_R)\big) = v + k\sqrt{{\rho_R\over\rho}}.
\ee
\end{lemma}
\begin{proof}
The Rankine-Hugoniot jump condition gives
\bel{R-H}
\aligned
& \sigma \big[   \rho\big]= \big[ \rho v\big],
\\
& \sigma \big[ \rho v \big ]=  \big[\rho(v^2+ k^2)\big],
\endaligned
\ee
where $\sigma $ denotes the speed of the discontinuity.
Consider first the 1-shock which should satisfy the Lax entropy inequality in the sense that
\[
\lambda (\rho_L, v_L) > \sigma > \lambda(\rho,  v ),
\]
for the 1-shock wave.
Eliminating the speed $\sigma$, we  obtain:
\[
v - v_L = -k \Big(\sqrt {\rho\over \rho_L} - \sqrt {\rho_L \over \rho}\Big), \quad  v > v_L.
\]
The form of the 2-shock wave follows from a similar calculation. The shock speeds can be obtained  directly from \eqref{lemma-shock}, \eqref{R-H}.
\end{proof}

%-----------------------------
\subsection{Standard Riemann problem }

We now consider the solution of the  standard Riemann problem of the homogenous Euler system \eqref{homo} associated with given initial data:
\bel{initial-Riemann-h}
U_0 (r) = \begin{cases}
U_L & 0<r < r_0, \\
U_R& r > r_0,
\end{cases}
\ee
where $r_0 > 0$  is a fixed radius and   $U_L = ( \rho_L ,  v_L ) $, $U_R =  (\rho_R, \rho_R)$ are constant states. To give the solution of the standard Riemann problem, we define now the {\bf 1-family-wave} and the {\bf 2-family wave}:
\bel{elementary}
 W_1^\rightarrow(U_L)= S_1^\rightarrow(U_L) \cup R_1^\rightarrow(U_L),
\qquad
W_2^\leftarrow(U_R)= S_2^\leftarrow(U_R) \cup R_2^\leftarrow(U_R),
\ee
where $ S_1^\rightarrow,  S_2^\leftarrow$ are 1 and 2-shocks while $ R_1^\rightarrow$, $ R_2^\leftarrow$ are 1 and 2-rarefaction waves.
It is obvious that if $U_L \in W_2^\leftarrow(U_R)$ or $U_R \in W_1^\rightarrow(U_L)$, then the Riemann problem is solved by the left state $U_L$ and the right state $U_R$ connected by either a 1-family wave or a 2-family wave. Otherwise, more analysis are required.

\begin{lemma}
\label{w-z}
On the $w-z$  plane where $w, z$ are the Riemann invariants of the Euler model  given by \eqref{Rie-in},  $ S_1^\rightarrow(U_L)$  defines a curve such that $
0 \leq {dw  \over dz  }  < 1,$
$ S_2^\leftarrow(U_R)$  defines a curve satisfying $
 0 \leq {dz \over dw }  < 1$ where $ S_1^\rightarrow,  S_2^\leftarrow$ are the 1 and 2-shocks given by \eqref{lemma-shock}.
\end{lemma}

\begin{proof}
Introduce functions $\Phi_\pm$:
\bel{phi}
\Phi_\pm (\gamma ) : = 1 + \gamma \bigg(1 \pm \sqrt{1+{2\over \gamma }}\bigg).
\ee
Taking  $\gamma  = \gamma  (v, v_L) = {(v - v_L)^2 \over 2k^2}$ along the 1-shock, we have
\[
\aligned
& w - w_L = v - v_L + k \ln {\rho \over \rho_L} = - \sqrt{2\gamma  k^2} + k \ln \Phi(\gamma ),
\\
& z - z_L = v - v_L - k \ln {\rho \over \rho_L} = - \sqrt{2\gamma  k^2} - k \ln \Phi(\gamma ).
\endaligned
\]
The tangent of the shock wave curve $S_1^\rightarrow(U_L)$ in the $w-z$ plane is given by
\[
 {dw \over dz} = {d(w - w_L) \over d(z - z_L) } = {d(w - w_L) \over d\gamma  }{d\gamma  \over d(z - z_L)}.
\]
Hence, we have  $0 \leq {dw \over dz} < 1$. A similar calculation gives the result of the 2-shock.
\end{proof}
Together with Lemma~\ref{w-z} and the form of elementary waves given in Lemmas~\ref{lemma-rarefaction}, \ref{lemma-shock}, some direct observations are given in order, concerning the standard Riemann problem of the homogenous Euler model \eqref{homo}:
\bei
\item For different given states $U_L , U_L'$, the two 1-family wave curves $W_1^\rightarrow(U_L) \cap W_1^\rightarrow(U_L')= \emptyset$. Similarly, for $U_R \neq U_R'$, the 2-family  wave curve $W_2^\leftarrow(U_R)$ has no intersection point with  $W_2^\leftarrow(U_R')$.
\item  The two families of wave curves cover the whole upper half $\rho-v$ plane as a result of Lemma~\ref{w-z}.
\item For given constant states $U_L, U_R$,  the waves  $W_1^\rightarrow(U_L)$ and   $W_2^\leftarrow(U_R)$ intersect once and only once at a point $U_M$.
\eei
We thus have the proposition:
\begin{proposition}[Solution of the standard Riemann problem]
Given two constant states  $U_L= (\rho_L, v_L)$ and $ U_R= (\rho_R, v_R)$, the  standard Riemann problem \eqref{homo}, \eqref{initial-Riemann-h} admits a unique entropic solution which only depends on ${r-r_0 \over t } $. More precisely, the solution is realized by the left state $U_L$, the right state $U_R$ and a uniquely defined  intermediate state $U_M$ where $U_L $ and $U_M$ are connected by a 1-wave while $U_M$ and $U_R$ are connected by a 2-wave.
\end {proposition}

%---------------------------------------------
\subsection{Wave interactions }
For the standard Riemann problem of the Euler model without source term \eqref{homo} with left-hand side constant state $U_L$ and right-hand side constant state $U_R$, define the  wave strength  of the Riemann problem $ \mathcal S = \mathcal S(U_L, U_R)$ :
\[
\mathcal S (U_L, U_R) : = |\ln\rho_L - \ln\rho_M| + |\ln\rho_R - \ln\rho_M|,
\]
where  $U_M$ is the unique intermediate state $U_M \in W_1^\rightarrow(U_L) \cap W_2^\leftarrow(U_R)$. We have the following lemma concerning $S$:

\begin{lemma}
\label{lemma-E}
Let $U_L$,  $U_P$, $U_R$ be three given constant states. The wave strengths associated with the  Riemann problem $(U_L, U_P), (U_P, U_R)$ and $(U_L, U_R)$ satisfy the following inequality
\bel{E}
\mathcal S (U_L, U_R) \leq \mathcal S(U_L, U_P) + \mathcal S(U_P, U_R).
\ee
\end {lemma}
To prove Lemma~\ref{lemma-E}, we first need the following calculation.

\begin{lemma}
\label{symmetric}
Given an arbitrary state $U_0$, the 1 and 2-shock wave curves $ S_1^\rightarrow(U_0)$ and  $S_2^\leftarrow(U_0)$ are reflectional symmetric with respect to the straight line parallel to $w=z$ passing the point $U_0$ on the $w-z$ plane where $w, z$ are the Riemann invariants of the Euler model introduced by \eqref{Rie-in}.
\end{lemma}
\begin{proof}
Denote by   $(w_0, z_0)$ the point $U_0$ on the $w-z$ plane. For a given point $(w,z)$ along the 1-shock, we have
\[
\Delta w_1: = w-w_0 = - \sqrt{2\gamma  k^2} + k \ln \Phi_+(\gamma ), \quad \Delta z_1 : = z -z_0 = - \sqrt{2\gamma  k^2} - k \ln \Phi_+(\gamma ),
\]
while  for a  point along the 2-shock  $(w,z)$:
\[
\Delta w_2 : = w -w_0 = - \sqrt{2\gamma  k^2} + k \ln \Phi_-(\gamma ), \quad \Delta z_2 :  = z -z_0 = - \sqrt{2\gamma  k^2} - k \ln \Phi_-(\gamma ),
\]
where the function $\Phi_\pm$ is defined by \eqref{phi}, which gives $\Phi_+(\gamma)\Phi_-(\gamma) = 1$. We have got the result by noticing that $\Delta w_1 = \Delta z_2, \quad \Delta z_1 = \Delta w_2$.
\end{proof}
We can thus continue the proof of Lemma~\ref{lemma-E}.
\begin{proof}[Proof of Lemma~\ref{lemma-E}]
Again, we stay on $w-z$ plane. From Lemmas~\ref {w-z}, \ref{symmetric}, we can see that the shock wavs $S_1^\rightarrow $,  $S_2^\leftarrow$ passing the same point $U_0$ are symmetric with respect to the straight line parallel to $w=z$ passing the point $U_0$. According to the definition of the wave strength \eqref{E} which is actually measured along the line $w=z$, the symmetry of waves gives immediately the result.
\end{proof}
%==================================================
\section{Fluid equilibria}
\label{sec:3}
\subsection{Critical smooth steady state solutions}
We now turn our attention to  steady state solutions $\rho=\rho(r), v=v(r)$, which satisfies the ordinary differential system:
\bel{steady-Euler}
\aligned
& {d\over dr}(r^2\rho v) = 0,
\\
&{d\over dr} \Big(r^2(v^2+k^2) \rho \Big) -2k^2 \rho r+m\rho = 0,
\endaligned
\ee
with the initial condition $\rho_0>0, v_0$ posed at a given radius $r=r_0 > 0$,
\bel{boundary-data}
\rho(r_0) =\rho_0 > 0, \qquad v(r_0) =v_0.
\ee
We call to \eqref{steady-Euler} the {\bf static Euler model}.
For a steady state solution $\rho= \rho(r), v=v(r)$, it is straightforward  to find a pair of algebraic relations:
\[
\aligned
& r^2\rho v=r_0^2\rho_0v_0,
\\
& {1\over2}v^2+k^2 \ln \rho -m{1\over r} = {1\over 2}v_0^2+k^2 \ln \rho_0-m {1\over{r_0}},
\endaligned
\]
from which we recover the equation for $v$ by eliminating $\rho$:
\bel{iden-v}
{1\over2}v^2-k^2 \ln \big(r^2 \sgn (v_0) v\big) -m{1\over r} = {1\over 2}v_0^2-k^2 \ln (r_0^2 |v_0|)-m {1\over{r_0}}.
\ee
Notice that once we get the value of $v$, we can have the value $\rho$ directly from the first equation of \eqref{steady-Euler}. Therefore, we focus on the analysis of the steady state velocity $v$.

 Introduce the function $G=G(r,v)$:
\bel{Fon-G}
G(r,v):= {1\over2}v^2-k^2 \ln (r^2  \sgn (v_0) v) -m{1\over r},
\ee
and we see  if $v=v(r)$ is a solution of \eqref{steady-Euler} with the  condition $v(r_0)=v_0$, then $G(r,v(r)) \equiv G(r_0,v_0)$ always holds. Differentiating $G$ with respect to $v$ and $r$, we obtain
\bel{derive-G}
\aligned
\del_v G & = v - {k^2\over v},
\qquad
\del_r G &  = {1\over r^2} ( m - 2 k ^2 r).
\endaligned
\ee
We can  immediately deduce the first-order derivative of the steady state velocity  $v=v(r)$:
\bel{derive-one}
{dv\over  dr} = {v\over  r^2} {2k^2 r -m \over v^2-k^2}.
\ee

It is obvious to see that
$\del_v G$=0 if and only if $v= \pm k$ while $\del_ r G =0$ if and only if $r= {m\over 2k^2 }  $ from \eqref{derive-G}.
This observation motivates us to find the steady state curves passing the points $( {m \over 2k^2 }, \pm k) $ on the $r-v$ plane $(0, +\infty)\times (-\infty, +\infty)$. We call the solution $ v= v(r)$  on the subset of $r-v$ plane $(0, +\infty) \times (-\infty, +\infty)$ {\bf the critical steady state solution} of the static Euler model \eqref{steady-Euler} if and only if satisfies $S(r, v(r))\equiv 0$ where $S=S(r, v )$ is given by
\bel{critical-S}
S(r,v ): = {1\over 2} v^2 - k ^2 \ln \big( {r ^2 |v | }\big) -m{1\over r}  +{ 3\over 2} k ^2 + k ^2 \ln {m ^2\over 4k^3}.
\ee
It is direct to check that $S ({ m\over 2k^2} , \pm k) =0$.
We now have the following lemma concerning the critical steady state curve.
\begin{proposition}
 \label{L:critical}
The static Euler model  \eqref{steady-Euler} admits four smooth  critical steady state curves  on the subset of $r-v$ plane $(0, +\infty) \times (-\infty, +\infty)$ denoted by $v_*^{P, \flat}, v_*^{P, \sharp},v_*^{N, \flat}, ,v_*^{N, \sharp}. $  Moreover, we have the following properties:
\bei
\item The sign of each solution does not change on the space domain $(0, +\infty)$.
\item On the interval $(0, {m \over 2k ^2})$, we have
\[
v_*^{N, \sharp}<- k< v_*^{N, \flat}<0 < v_*^{P, \flat}<k < v_*^{P, \sharp},
\]
while on the interval $({m \over 2k ^2}, +\infty)$, we have
\[
v_*^{N, \flat}<- k< v_*^{N, \sharp}<0 < v_*^{P, \sharp}<k < v_*^{P, \flat}.
\]
\item The solutions $v_*^{N, \sharp},  v_*^{N, \flat}$ intersect once at $({m \over 2k ^2}, - k)$ while  $v_*^{P, \sharp},  v_*^{P, \flat}$ intersect once at $({m \over 2k ^2}, k)$.
\item The derivatives of each solution at $({m \over 2k ^2}, \pm k)$ are give by
\bel{derivatives-k}
 {dv_*^{P, \sharp}  \over dr }   ({m\over 2 k ^2})= {d v_*  ^{N, \flat}  \over dr } ({m\over 2 k ^2})=- {2k^3\over m},\quad
  {dv_*^{P, \flat}  \over dr }   ({m\over 2 k ^2})= {d v_*  ^{P, \sharp}  \over dr } ({m\over 2 k ^2})=  {2k^3\over m}.
\ee

\eei

\end{proposition}

\begin{proof}
We would like to show that for every fixed radius $r> 0 $ and $r\neq {m\over 2k^2}  $, there exists four different values $v$ satisfying \eqref{critical-S}.  Observing $S(r,v)= S(r, -v)$, we  first consider the case where $v>0$.  According to \eqref{derive-G}, for every fixed $r>0, $ $S(r, \cdot)$ reaches its minimum at $v= k$ and the value is given as
\[
S^k(r): = 2 k^2 - k^2 \ln r ^2k^2 -{ m\over r} + k ^2 \ln {m ^2 \over 4k^3}.
\]
Since $ \del_ r S ^k ={1\over r^2} ( m - 2 k ^2 r)$, we have $S^k(r)< S^k ( {m\over 2k^2} )=0$ .  Moreover, we have $\lim  \limits_{v \to   0} S (r,v)=+ \infty$ and $\lim \limits_{v \to + \infty } S (r,v) = + \infty$. Therefore, for every fixed $r \neq {m\over 2k^2}$, $S(r, v )$ admits two different positive roots $v_1\leq k \leq v_2$ on $(0, + \infty )$ where the equality holds only once at the point $r=  {m\over 2k^2}$. The symmetry of  $S(r, \cdot )$ with respect to $v= 0$ gives two other negative roots  $v_3\leq -k \leq v_4$.

Since $S_v \neq 0$ when $v\neq \pm k$, there exist four smooth different solutions on the interval $(0, {m\over 2k^2})$ and $( {m\over 2k^2}, +\infty)$ respectively. To extend the steady solution on the whole domain $(0, +\infty)$, we have to treat the very points $( {m\over 2k^2}, \pm k)$. Indeed, we have, by the  L'H\^ opital's rule, $
 {d v \over dr} \big( {m\over 2k^2}\big) = {k\over (m/ 2k^2)^2
} {  k^2  } \Big/\Big(k  {d v \over dr}\big( {m\over 2k^2}\big)\Big)
$, which gives
\bel{D-s}
 {d v \over dr} \Big( {m\over 2k^2}\Big)= \pm{2k^3 \over m},
\ee
whose sign depends on the choice of the branch of curves. According to \eqref{D-s}, we are able to  to keep the solution smooth on the whole domain $(0, +\infty)$ by keeping the sign of the derivative of $v$ at $r= {m\over 2k^2}$. We thus define  the four different solutions on $(0, +\infty)$:
\bel{four-critical}
\aligned
& v_*^{P, \flat}(r)= \begin{cases}
v_1(r) &  r \in (0,{m\over 2k^2} ), \\
v_2 (r) & r  \in ({m\over 2k^2}, +\infty ),
\end{cases}
\qquad
v_*^{P, \sharp}(r)= \begin{cases}
v_2(r) & r \in (0,{m\over 2k^2} ), \\
v_1 (r) & r  \in ({m\over 2k^2}, +\infty ),
\end{cases}
\\
& v_*^{N, \flat}(r)= \begin{cases}
v_3(r) &  r \in (0,{m\over 2k^2} ), \\
v_4 (r) & r  \in ({m\over 2k^2}, +\infty ),
\end{cases}
\qquad
v_*^{N, \sharp}(r)= \begin{cases}
v_4(r) &  r \in (0,{m\over 2k^2} ), \\ 
v_3 (r) & r  \in ({m\over 2k^2}, +\infty ).
\end{cases}
\endaligned
\ee
The derivative of the velocities in \eqref{derivatives-k} follows directly from \eqref{D-s} and \eqref{four-critical}.
\end{proof}
%--------------------------------------------------------------
\subsection{Families of steady state solutions}
The former construction gives that the relation $S(r,v) \equiv 0$ admits four different solutions on the whole domain $(0, +\infty)$. We would like now to give all families of solutions according to the sign of $S(r,v)$ defined in \eqref{critical-S}.
We now study general cases of the steady state solutions.

 We then have the following lemma.
\begin{lemma}
Let $S= S(r,v)$ be the function defined by \eqref{D-s}, then:
\bei
\item If $S=const.> 0$, then there exists four solutions $v=v(r)$ satisfying the algebraic equation \eqref{iden-v} on the whole space interval  out of the black hole $(0, +\infty)$.
\item If $S=const.<  0 $, then there exist  two radius $ 0  <\underline r_ S< { m  \over 2 k ^2} <\bar r_ S$ such that  there exist four solutions $v=v(r)$  satisfying the algebraic equation \eqref{iden-v} on the  interval $(0,  \underline  r_S )$ and four solutions satisfying \eqref{steady-Euler} on the  interval $( \bar r_S, +\infty )$.
\eei
\end{lemma}

\begin{proof}
We now focus on the case where $S=const.> 0$. Again, $S(r,v)= S(r, -v)$ allows us to consider the case where $v>0$. Now we notice that   $G(r,v)- G({ m \over 2 k ^2}, k )=S (r,v)$ where $G$  is defined by \eqref{Fon-G}.
By the formula of \eqref{derive-G},  for  all the  fixed $r\in (0, +\infty)$, the equation $G(r,v)- G({ m \over 2 k ^2}, k )=const. >0  $ admits two positive roots $v_S ^{P, \sharp}>k > v_s ^{P, \flat}$ if and only if $G(r, k ) <  G({ m \over 2 k ^2}, k )$. Moreover,  \eqref{derive-G} gives the fact that $G (r, k )$  reaches its  maximum at the point $r={m \over 2 k ^2} $ and we thus have $G (r,k) <  G({ m \over 2 k ^2}, k  )$.  We have another two negative roots $v_\gamma^{N,\sharp}<-k< v_\gamma^{N,\flat}$ following from the same analysis.

Now if $S=const.<  0$,  there exist two points  $ 0 <\underline r_ S< { m  \over 2 k ^2} <\bar r_ S$ such that $S (\underline r_ S, k )= S (\bar r_ S, k )= 0$ and  $S (r, k )< 0$ for all $r\in (\underline r_ S, \bar r_ S)$.  We have four roots satisfying \eqref{iden-v} among which two are defined  only on $(0,\underline r_ S)$ while two on $( \bar r_ S, +\infty)$ respectively.
\end{proof}

We can now give the existence result of the steady state solution of the Euler model \eqref{Euler1}.
\begin{theorem}[Families of steady state solutions]
\label{steady-state}
Consider the family of steady state  solutions of the Euler  model
 \eqref{steady-Euler}. Then, for any given radius $r_0 >0 $ , the density $\rho_0>0$ and the velocity $v_0 $, we have: there exists
 a unique smooth steady state solution $\rho=\rho(r), v = (r)$  satisfying \eqref{steady-Euler}  together with the initial condition $\rho_0= \rho(r_0), v(r_0) = v_0$
such that the velocity satisfies $\sgn(v) = \sgn(v_0)$ and   $\sgn(|v|- k ) = \sgn(|v_0|- k)$ on the corresponding domains of definition. Furthermore, we have different families of solutions:
\bei
\item If   $  G(r_0, v_0)  >   -{ 3\over 2} k ^2 -  k ^2 \ln {m ^2\over 4k^3} $ in which the parameter $G = G(r,v) $ was introduced in \eqref{Fon-G}, then
 the steady state solution is defined on the whole space interval  $(0, +\infty)$.
\item If   $  G(r_0, v_0)  =    -{ 3\over 2} k ^2 -  k ^2 \ln {m ^2\over 4k^3} $, then we have the critical steady state solution on the whole interval $(0, +\infty)$ whose formula is given by \eqref{four-critical}.
\item  If   $ G(r_0, v_0)   <     -{ 3\over 2} k ^2 -  k ^2 \ln {m ^2\over 4k^3} $, then the solution is defined on $(0,\underline r_S )$ if $r_0 < {m \over 2  k ^2} $ or $(\bar  r_S, +\infty)$ if $r_0> {m \over 2  k ^2}$ where $\underline r_S, \bar r_S$  satisfies $G (\underline r_S,k)=G(\bar  r_S, k) = G(r_0, v_0)$.
\eei

\end{theorem}

\begin{figure}[htbp]
\label{LIMIT}
\centering
\epsfig{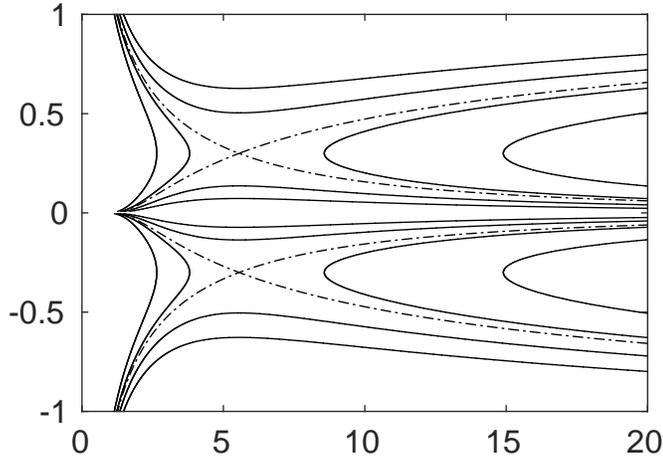}
\caption{Plot of steady state solutions. }
\end{figure}
%--------------------------------------
\subsection{Steady shock}
We now consider the steady shock which is also a solution of the static Euler equation \eqref{steady-Euler} but contains one discontinuity satisfying also the entropy condition. We give the following lemma.

\begin{lemma}[Jump conditions for steady state solutions]
\label{steady-shock}
A steady state discontinuity of the Euler model \eqref{Euler1} associated with left/right-hand limits $(\rho_L , v_L)$ and $(\rho_R, v_R)$ must satisfy
\[
{\rho_R\over \rho_L} =   {v_L^2 \over k^2}.
\qquad v_L \, v_R= k^2, \qquad v_ L \in (-k, 0) \cup (k, +\infty).
\]
\end{lemma}
\begin{proof}
From the \emph{steady Rankine-Hugoniot relations}
\[
\big[ \rho v \big] = 0, \qquad
 \big[  \rho(k ^2 + v^2 ) \big] = 0,
 \]
 where the bracket $[\cdot]$ denoted the value of the jump
 and we deduce that
\[
 \rho_R  v_R = \rho_L v_L,
\qquad
 \rho_R(v_R ^2+ k^2 )  =\rho_L  (v_L^2+ k^2) ,
\]
which gives the relation of the left-hand side and the right-hand side limit of the jump. Then  the Lax entropy condition requires that $\lambda(\rho_L, v_L ) > 0 > \lambda(\rho_R, v_R ) $,   $\mu(\rho_L, v_L ) > 0  > \mu(\rho_R, v_R ) $ for 1 and 2-waves.
 \end{proof}

 Lemma~\ref{steady-shock} permits us to construct a \emph{steady shock wave}  of the Euler model \eqref{Euler1} with a  zero speed, that is,  a function composed of a pair of steady state solutions $(\rho_L, v_L)=(\rho_L, v_L) (r), (\rho_R, v_R)=(\rho_R, v_R) (r)$  separated  by a discontinuity at a fixed point $r_0$ with the relation
\bel{0-family}
v_R(r_0) ={ k ^2 \over v_L (r_0)} , \qquad  \rho_R(r_0 ) =  {v_L(r_0)^2 \over k^2} \rho_L (r_0),
\ee

with
\bel{entropy-0}
v_L (r_0) \in   v_ L \in (-k, 0) \cup (k, +\infty).
\ee

%=============================
\section{The generalized Riemann problem}
\label{sec:4}
\subsection{The rarefaction regions}
The {\bf generalized Riemann problem}  of the Euler model is a Cauchy problem of \eqref{Euler1} with given initial data given as
\bel{initial-Riemann}
U_0 (r) = \begin{cases}
U_L (r) & \underline r <r < r_0, \\
U_R (r)  & r_0<r  <   \bar r ,
\end{cases}
\ee
for a fixed  radius $r_0 > 0$ and two steady state solutions $U_L = (\rho_L, v_L)$ and $U_R= (\rho_R, v_R)$ such that the static Euler equation \eqref{steady-Euler} holds.

For simplicity, we write $(\rho_L , v_L)(r_0)= (\rho_L^0 , v_L^0) = U_L^0 $ and $(\rho_R , v_R)(r_0)= (\rho_R^0 , v_R^0)=U_R^0$.
To solve the generalized Riemann problem, we need first to fix the point $r=r_0$ and solve the standard Riemann problem \eqref{homo} with initial data
\[
U_0  (r)=\begin{cases}
U_L ^0 &  \underline r  <r<r_0,
\\
U_R^0 & r_0 <r  < \bar r.
\end{cases}
\]
The  standard Riemann problem at a fixed radius is solved by three constant states $ U_L^0= (\rho_L^0 , v_L^0)$, $U_M^0= ( \rho_M^0, v_M^0) $ and $U_R^0= (\rho_R^0 , v_R^0)$ connected to each other with 1-wave and 2-wave respectively where the intermediate constant state is  given by
 \bel{UM0}
U_M ^0  \in  W_1^\rightarrow (U_L  ^0)\bigcap W_2 ^\leftarrow (U_R^0).
\ee

Coming back to the Euler equation with source term \eqref{Euler1},  we would like to construct a solution of the generalized Riemann problem \eqref{Euler1}, \eqref{initial-Riemann} with three steady state solutions connected by generalized elementary curves. We give the intermediate steady state solution denoted by $(\rho_M , v_M)=(\rho_M , v_M)(r )$  by the static Euler equation \eqref{steady-Euler} with initial data $(\rho_M ^0, v_M^0)$ at the point $r=r_0$, that is
\bel{Um-middle}
(\rho_M , v_M)(r_0  )= (\rho_M^0 , v_M^0).
\ee

To work on different types of elementary waves, we consider the following differential equations:

\bel{wavecurves-left}
\aligned
 & {d {r_L^M}^+  \over dt }=
\begin{cases}
\lambda\big(\rho_M( {r_L^M}^+), v_M( {r_L^M}^+)\big), & v_L^0< v_M^0, \\
\sigma_1\Big(\big(\rho_L({r_L^M}^+), v_L({r_L^M}^+)\big), \big(\rho_M({r_L^M}^+), v_M({r_L^M}^+)\big)\Big), & v_L ^0 >v_M^0,
\end{cases}
\\
 & {d {r_L^M}^-  \over dt }= \begin{cases}
 \lambda\big(\rho_L({r_L^M}^-, v_L({r_L^M}^-)\big), & v_L^0< v_M^0,
\\
 \sigma_1\Big(\big(\rho_L({r_L^M}^-), v_L({r_L^M}^-)\big), \big(\rho_M({r_L^M}^-), v_M({r_L^M}^-)\big)\Big), &  v_L^0> v_M^0,
\end{cases}
\\
& {r_L^M}^\pm (0)= r_0,
\endaligned
\ee
as well as
\bel{wavecurves-right}
\aligned
 & {d {r_M^R}^+  \over dt }=
\begin{cases}
\mu \big(\rho_M( {r_M^R}^+), v_M( {r_M^R}^+)\big), & v_M^0< v_R^0, \\
\sigma_2 \Big(\big(\rho_L({r_M^R}^+), v_L({r_M^R}^+)\big), \big(\rho_M({r_M^R}^+), v_M({r_M^R}^+)\big)\Big), & v_M ^0 >v_R^0,
\end{cases}
\\
 & {d {r_M^R}^-  \over dt }= \begin{cases}
\mu\big(\rho_L({r_M^R}^-, v_L({r_M^R}^-)\big), & v_M^0< v_R^0,
\\
 \sigma_2\Big(\big(\rho_L({r_M^R}^-), v_L({r_M^R}^-)\big), \big(\rho_M({r_M^R}^-), v_M({r_M^R}^-)\big)\Big), &  v_M^0> v_R^0,
\end{cases}
\\
& {r_M^R}^\pm (0)= r_0,
\endaligned
\ee
where $ \sigma_1,   \sigma_2 $ are speeds of 1 and 2-shocks respectively and $\lambda, \mu$ are eigenvalues given by \eqref{eigenvalue}.

\begin{lemma}
\label{well-defined-cruves}
Let $(\rho_L, v_L)= (\rho_L, v_L)(r), (\rho_R, v_R)= (\rho_R, v_R)(r)$ be two steady state solutions given by \eqref{steady-Euler}. The  curves  ${r_L^M}^\pm,  {r_M^R}^\pm$ are uniquely defined by \eqref{wavecurves-left}, \eqref{wavecurves-right} for all $t>0$ respectively,  with  bounded derivatives.
\end{lemma}

\begin{proof}
 We first consider the 1-wave.
If $(\rho_L^0, v_L^0)$ and $(\rho_M^0, v_M^0)$ are connected by a 1-rarefaction, then we have
\[
 {d {r_L^M}^+  \over dt }=
\lambda\big(\rho_M( {r_L^M}^+), v_M( {r_L^M}^+)\big),  \qquad
{d {r_L^M}^-  \over dt }=\lambda\big(\rho_L({r_L^M}^-, v_L({r_L^M}^-)\big).
\]
Following from the existence theory of ordinary differential equations, there exists a time $T>0$ such that the curves are well-defined on $0<t<T$. To prove that these curves are indeed defined globally in time, we have to show that steady state solutions can not be sonic along the wave curves, referring to Theorem~\ref{steady-state}.   We take into account two cases:
 \bei
 \item When  $r_0<  { m\over 2k ^2 } $,  $v_L= v_L (r)$ cannot be sonic for all $\bar r <r< r_0$. Then we only have to consider the case where ${r_L^M}^-(t)> r_0$,  which gives$ {d{r_L^M}^-(t)\over dt } >  0$ ,
providing $v_L \geq k$. If there exists a finite time $t_1$ such that $v_L \big({r_L^M}^-(t)\big) = k$ , then ${d{r_L^M}^-(t)\over dt } | _{t=t_1}=  v_L \big(\ {r_L^M}^-(t_1)\big) -k$, which provides a contradiction.
 \item When $r_0 \geq  { m\over 2k ^2 } $,   the following equality holds $r_L^* <\underline r<r_0  $ where $r_L^*$ is the sonic point of $(\rho_L,v_L)$, then we have at once the result.
  \eei

  Now if $(\rho_L^0, v_L^0)$ and $(\rho_M^0, v_M^0)$ is connected by a 1-shock,  the result will hold if $(\rho_L, v_L)$ will not reach to the sonic point on $\big(\bar r, {r_L^M}^- (t)\big)$ for $0<t<T$. We consider the two cases as follows.
  \bei
\item  When  $r_0<  { m\over 2k ^2 } $, we only have to consider the case where $\sigma_1>0$. The entropy condition gives $\lambda(\rho_L, v_L) > \sigma_1>  \lambda(\rho_M ,v_M) $, leading to $v_L >  k $. Then we have the result.
\item When $r_0 \geq  { m\over 2k ^2 } $,   we have $r_L^* <\underline r<r_0  $ and the result holds.   \eei

 A similar calculation gives all the curves listed in the lemma.
\end{proof}

It follows directly from the definition that ${r_L^M}^-(t)\leq  {r_L^M}^+(t)\leq  {r_M^R}^-(t) \leq  {r_M^R}^+ (t)$, which permits us to define  five disjoint regions below for all fixed $t>0$:
$\big(\underline r , {r_L^M}^-(t)\big)$, $\big({r_L^M}^-(t), {r_L^M}^+(t)\big)$, $  \big({r_L^M}^+(t),  {r_M^R}^-(t)\big)$, $\big({r_M^R}^-(t),  {r_M^R}^+(t)\big)$,  $\big({r_M^R}^+(t), \bar r \big)$  and we denote by $\big({r_L^M}^-(t), {r_L^M}^+(t)\big)$ and $\big({r_M^R}^-(t),  {r_M^R}^+ (t)\big)$ the \emph{1-rarefaction region} and the  \emph{2-rarefaction region}.

%----------------------------------------------------------------------------------
\subsection{Exact solution to Riemann problem}

We now give  the solution $U= (\rho, v )=(\rho, v)(t, r)$ for the generalized Riemann problem. Write

\bel{Riemann-sol}
U (r) : = \begin{cases}
U_L (r)& \underline r < r < {r_L^M}^-(t)􏰁, \\
 \tilde U_1(t, r)& {r_L^M}^-(t)􏰁< r <{r^M_L}^ +(t), \\
U_M  (r)& 􏰀{r^M_L}^ +(t)􏰁<r < {r^R_M}^ -(t), \\
\tilde U_2(t,r) &  􏰀{r^R_M}^ -(t)􏰁< r < {r^R_M}^ +(t),\\
U_R (r)&  {r^R_M}^ +(t)<r<  \bar r,
 \end{cases}
\ee
where ${r^M_L}^ \pm$, ${r^R_M}^ \pm$ are boundaries of the rarefaction regions defined by \eqref{wavecurves-left},  \eqref{wavecurves-right}. Here, $U_L =  (\rho_L,  v_L)$,  $U_M = (\rho_M,  v_M)$, $U_R= (\rho_R,v_R)$ are three steady state solutions and $\tilde U_1$ and  $\tilde U_2$  are  generalized rarefaction waves to be given  by the  integro-differential problem following from Liu~\cite{Liu2}. Indeed, we give the function $\tilde{U}_j (t, \theta_j) = (\tilde \rho_j, \tilde v_j )(t, \theta_j)$, $j =1,2 $  and the new variable $\tilde r = \tilde r(t, \theta_j)$.  To seek for the form of $\tilde U_j$ and $\tilde r$,  we consider the following problem:
\bel{form-of-U&r}
\aligned
& \partial_{\theta_j} \tilde r \partial_t \tilde U_j + \bigg (\partial_U F(\tilde U_j) - \lambda(\tilde U_j) \bigg) \partial_{\theta_j} \tilde U_j = S(\tilde U_j) \partial_{\theta_j} \tilde r,\\
& \partial_t \tilde r = \lambda(\tilde U_j(t, \theta_j)),
\endaligned
\ee
with boundary and initial conditions reading
\bel{BIC}
\aligned
& \tilde{U}_j( t, \theta_j^0) = U_k^0(\tilde r(t, \theta_j ^0)),  \quad \tilde U_j(0, \theta_j) = h_1(\theta_j),\\
 & \partial_t \tilde r(t, \theta_j^0) = \lambda(U_k^0(\tilde r)),\quad  \tilde r(0, \theta_j) = r_0,
\endaligned
\ee
where we give $\theta_j^0 = \lambda (U_k^0)$, $j=1, 2$, $k=L, R$ and the function $h_j $ defined by
\bel{if}
\xi = \lambda_j (h_j (\xi)) = {r-r_0 \over t},
\ee
where $\lambda_1=\lambda, \lambda_2=\mu$ are the eigenvalues of the 1 and 2 families.
\begin{lemma}
The integro-differential problem \eqref{form-of-U&r}, \eqref{BIC} admits a unique $\tilde U_ j $ smooth for all fixed  time $t>0$. \end{lemma}
\begin{proof}
To prove the lemma, we use a standard fixed point argument. Without loss of generality, we consider the 1-rarefaction wave. Denote by $l_1, l_2$ two linearly independent vectors corresponding to  $\lambda, \mu$ respectively. Multiplying \eqref{form-of-U&r} by $l_2$, we have
\[
\aligned
& D V_2 = { \partial_{\theta_2}\tilde r\over \mu - \lambda } l_2 \cdot S + Dl_2 \cdot V_1 ,\\
 & \partial_t V_1 = l_2 \cdot S + \partial_t l_2 \cdot V_1,
 \endaligned
\]
where we have defined
$V_1 = l_1 \cdot \tilde U_1$, $  V_2 = l_2 \cdot \tilde U_1,$ and the operator reads $ D = {\partial_{\theta_2}\tilde r\over \mu - \lambda } \partial_t + \partial_{\theta_2}$ whose integral curves  starting from $(\tau, \lambda(U_0))$  is  denoted by $\zeta$.
We thus have
\bel{T}
\aligned
&V_2(t, \theta_1) = V_2(\tau , \lambda(U_0)) + \int_\zeta \bigg({ \partial_{\theta_2}\tilde r\over \mu - \lambda } l_2 \cdot S + Dl_2 \cdot V_1 \bigg)d\theta_1,\\
&  V_1(t, \theta_1) = V_1(0, \xi) + \int_{0}^t\bigg( l_2 \cdot S + \partial_t l_2 \cdot V_1 \bigg)d\theta_1.
\endaligned
\ee
Now  let  $\mathcal F$  be the operator of the right-hand side of \eqref{T} and we study the iteration method  $\tilde U_1^{(l)} = \mathcal F^{(l)} \tilde U_1^0$, $l\geq 1$ where $\tilde U_1^0$  is an  arbitrary smooth function satisfying the initial-boundary condition
$\tilde U_1^0(t, \theta_j^0) = \tilde U_1(t, \theta_j)$ ,$\tilde U_1^0(0, \theta_j) = \tilde U_1(0, \theta_j)$.  It is  direct to  check that for sufficiently small $t_1$, $\mathcal F$ is contractive in the max norm of $\tilde U_j^0$. By iterating the operator $\mathcal F$, we prove that there exists a unique solution $\tilde U_1$ for all $0<t\leq \Delta  t_1$. Then we repeat the process by taking $\tilde U_1 (t_1, \cdot)$ as initial condition and there exists a time $\Delta t_2$ such that $\tilde U_1$ is defined by all $\Delta t_1<t<\Delta t_1+ \Delta t_2$ and it is directly to see that $\Delta t_1\leq \Delta t_2 $ by the definition of the operator $\mathcal F$.
\end{proof}

According to the construction above, we conclude  the following theorem.
\begin{theorem} [The solution of the generalized Riemann problem]
Consider the generalized Riemann problem for the Euler model \eqref{Euler1},\eqref{initial-Riemann}. There exists a  weak solution to the generalized problem on $t>0$ whose exact form is given by \eqref{Riemann-sol}, satisfying the Rankie-Hugoniot jump condition and the Lax entropy condition.
\end{theorem}

%--------------------------------------
\subsection{Evolution of total variation}
It is obvious that the total variation of $\ln \rho$ of the solution of the standard Riemann problem \eqref{homo}, \eqref{initial-Riemann-h} stays as a constant when time passes. However, it is a different story for the generalized Riemann problem \eqref{Euler1}, \eqref{initial-Riemann}.  We have the following lemma.
\begin{lemma}
\label{TV-R}
Let $U= (\rho, v)= (\rho, v)(t,r)$ be the solution of the generalized Riemann problem of the Euler model \eqref{Euler1} whose initial data $U_0= (\rho_0, v_0)=(\rho_0, v_0)(r)$ has the form \eqref{initial-Riemann} . Then we have
\bel{TV-Riemann}
TV_{[\underline r, \bar r]}\big(\ln\rho (t, \cdot )\big ) < TV_{[\underline r, \bar r]} \big(\ln\rho(0+, \cdot )  \big)\big(1+ O(t )\big),
\ee
for all $t>0$.
\end{lemma}

\begin{proof}
Let $U_M= U_M(r)$ be the intermediate steady state solution associated with the left state $U_L $and the right state $U_R$ given in the initial data. According to \eqref{wavecurves-left}, we have
\[
\aligned
U_L ({r_L ^M} ^- (t))-U_M ({r_L ^M} ^- (t))=&  U_L (r_0)-U_M(r_0)+|U_L (r_0)-U_M(r_0)|O ({r_L ^M} ^- (t)-r_0 )\\
=& U_L (r_0)-U_M(r_0)+|U_L (r_0)-U_M(r_0)|O (t).
\endaligned
\]
Moreover, according to the construction of the generalized Riemann problem, we give
\[
\aligned
& TV_{[\underline r, \bar r]} \big(\ln\rho(t+, \cdot )  \big)- TV_{[\underline r, \bar r]}\big(\ln\rho (0+, \cdot )\big )\\
 \leq&  \big(\ln\rho_L(r_0)- \rho_M(r_0)|+ \ln\rho_L(r_0)- \rho_M(r_0)|\big)O (t)
=   TV_{[\underline r, \bar r]}\big(\ln\rho (0+, \cdot )\big )O (t),
\endaligned
\]
where we have used the continuous dependence property $|U_L (r_0)-U_M(r_0)|= O(1)|(\ln\rho_L(r_0)- \rho_M(r_0)|$. This ends the proof of the lemma.
\end{proof}
%=================================
\section{Triple Riemann problem}
\label{sec:5}
\subsection{Preliminary}
Considering the fact that a steady state solution of the steady Euler model \eqref{steady-Euler} may not be defined globally as is the result of Theorem~\ref{steady-state} and we are obliged to introduce the {\bf triple Riemann problem} in order to complete the Glimm method in the coming section, that is, a Cauchy problem associated with initial data composed of three steady state solutions:
\bel{initial-triple}
U_0  (r)=
\begin{cases}
U_\alpha (r) & \underline r <r < r_s, \\
U_\beta (r)& r_s<r<r _b, \\
U_\gamma (r) & r_b<r<\bar r,
\end{cases}
\ee
for fixed radius $0<\underline r < r_1<r _2<\bar r $ and steady states $U_\alpha= (\rho_\alpha, v_\alpha) $,  $U_\beta= (\rho_\beta, v_\beta) $,  $U_\gamma= (\rho_\gamma, v_\gamma) $. We denote by $U_\alpha(r_s)= U_\alpha^s= (\rho_\alpha^s, v_\alpha^s) $,  $U_\beta(r_s)= U_\beta^s= (\rho_\beta^s, v_\beta^s) $,  $U_\beta(r_b)= U_\beta^b= (\rho_\beta^b, v_\beta^b) $,  $U_\gamma(r_b)= U_\gamma^b= (\rho_\gamma^b, v_\gamma^b) $.

We first give the main conclusion of this section:
\begin{theorem}
\label{Triple}
Consider a given  initial data composed of three steady state solution $U_\alpha ,U_\beta, U_\gamma $. Then for all $t>0$, the  triple Riemann problem of the Euler model \eqref{Euler1}, \eqref{initial-triple} admits a weak solution $U= (\rho, v)= (\rho, v)(t, r)$ such that for all  $t>0$, we have:
\bel{TV-triple}
TV_{[\underline r, \bar r]}\big(\ln\rho (t, \cdot )\big ) < TV_{[\underline r, \bar r]} \big(\ln\rho(0+, \cdot )  \big)\big(1+ O(\bar r- \underline r)\big).
\ee
\end{theorem}
We define the \emph{left-hand problem} as a generalized Riemann problem with initial data
\[
U_0  (r)=
\begin{cases}
U_\alpha (r) & r < r_1, \\
U_\beta (r)& r> r _1, \\
\end{cases}
\]
and  the \emph{right-hand problem} as a generalized Riemann problem with initial data
\[
U_0  (r)=
\begin{cases}
U_\beta (r)& r<r _2, \\
U_\gamma (r) & r> r _2 ,
\end{cases}
\]
Since the Euler model \eqref{Euler1} is strictly hyperbolic following from Proposition~\ref{hyperbolicity},  for a small enough time $t>0$, both the left-hand and the right hand problem admit a solution denoted by $U_L =U_L(t, r)$ and $U_R =U_R(t, r)$ respectively and  the wave curves of the solutions do not interact. We denote by  ${r_L^M} ^\pm_L  $, ${r_M^R } ^\pm_L  $  the rarefaction regions  boundaries of the left-hand side problem and  ${r_L^M } ^\pm_R $, ${r_M^R } ^\pm_R  $ of the right-hand side  problem \eqref{wavecurves-left}, \eqref{wavecurves-right}. We then define the moment of the first interaction denoted by $T_f$:
\bel{first}
T_f :=\sup\{t>0|{r_M^R } ^+_L(t) \leq {r_L^M } ^-_R(t)  \}.
\ee
Clearly, if $T_f= +\infty$, the triple Riemann problem \eqref{Euler1}, \eqref{initial-triple} exists a solution reading
\bel{before-T_f}
U^f (t, r)= \begin{cases}
U_L(t, r)& \underline r <r<r _2, \\
U_R (t, r) & r_2 <r< \bar r .
\end{cases}
\ee
%------------------------------------------
\subsection{Possible interactions}
If the moment of the first interaction $T_f<+\infty$, then the waves of the left and the right-hand Riemann problem did have interactions. Possible interactions are given in order:
\bei
\item 2-shock of the left-hand problem and 1-shock of the right-hand problem,
\item 2-shock of the left-hand problem and 1-rarefaction of the right-hand problem,
\item 2-rarefaction of the left-hand problem and 1-shock of the right-hand problem,
\eei
which are denoted by Problems $P-ss$, $P-sr$, $P-rs$ respectively.  For later use, we denote by $U^{\alpha,\beta}_M$, $U^{\beta,\gamma}_M$ the intermediate states of the left and right-hand problems respectively.
We consider different kinds of interactions separately.
\begin{lemma}\label{ss}
If $T_f<+\infty$ where $T_f$ is defined by \eqref{first} and we have the 2-shock of the left-hand problem and the 1-shock of the right-hand problem of the Euler model \eqref{Euler1}, then there exists a time $T_{ss}$ such that Problem $P-ss$ admits a solution on $0<t< T_{ss}$.
\end{lemma}

\begin{proof}
We only have to consider the solution after $t>T_f$.  We denote by $U_M^{ss}= U_M ^{ss}(t, r)$ the solution of the generalized problem with initial states $U^{\alpha,\beta}_M, U^{\beta,\gamma}_M$ separated by $r= {r_L^M} ^+_L(T_f)= {r_M^R} ^-_R(T_f)$ at $t= T_f$. Then for $T_f<t<T_{ss}$, we give
\bel{sol-Tss}
U^{ss}(t, r)=  \begin{cases}
U_L(t, r)& \underline r <r<{r_L^M } ^+ _L(t) , \\
U_M^{ss} (t, r) & {r_L^M } ^+ _L (t) <r<{r_M^R} ^-_R(t), \\ 
U_R(t, r) & {r_M^R} ^-_R(t)< r< \bar r,
\end{cases}
\ee
where
\bel{T_ss}
T_{ss}= \min \left(\sup \{t>T_f|{r_L^M} ^-_M (t)>{r_L^M} ^+_L (t) \}, \sup \{t>T_f|{r_M^R} ^-_R(t)>{r_M^R} ^+_M(t) \} \right), 
\ee
where ${r_L^M} ^\pm_M$  are boundaries of the rarefaction regions of  the state $U_M^{ss}$ given by \eqref{wavecurves-left}, \eqref{wavecurves-right}. Thus Problem P-ss admits a solution for all $t<T_{ss}$.
\end{proof}

We now consider Problem $P-rs$.
\begin{lemma}
\label{rs}
Let $T_f$ be the first moment of interaction and we suppose $T_f< +\infty$ and the Euler model  \eqref{Euler1} has  2-rarefaction of the left-hand problem and the 1-shock of the right-hand problem. Then there exists a time $T_{rs}$ such that we have a solution of Problem $P-rs$ for all $0<t<T_{rs}$.
\end{lemma}

\begin{proof}
Again, we only have to construct a solution after $t>T_f$. Let us first write $\tildeU^{\alpha,\beta}_2=\tildeU^{\alpha,\beta}_2(t, r)  $ the 2-rarefaction wave of the left-hand problem which evolves in  the region $\big({r_L^M} ^-_L(t), {r_L^M} ^+_L(t)\big)$. Then we give
\bel{sol-Tsr}
U^{rs}_0 (t, r)=  \begin{cases}
U_L(t, r)& \underline r <r<{r_M^L } ^+ _L(t)  , \\
U_M^{rs} (t, r) & {r_M^L } ^+ _L(t)  <r<{r_M^R} ^-_R(t), \\
U_R(r) & {r_M^R} ^-_R(t)< r< \bar r,
\end{cases}
\ee
where the function $U_M^{rs} (t, r) $ is  given by
\bel{sol-TsrM-0}
U^{rs,0}_M (t, r)=  \begin{cases}
\tildeU^{\alpha,\beta}_2 (t, r) & {r_M^L } ^+ _L(t)<r< \widetilde { r_M^L } ^- _{rs}(t) , \\
\tilde U^{rs}_1(t, r) & \widetilde { r_M^L } ^- _{rs}(t)<r< \widetilde { r_M^L } ^+ _{rs}(t),\\
U_{MM} ^{rs,0} (r) & \widetilde { r_M^L } ^+ _{rs}(t)  <r< \widetilde {r_M^R} ^-_{rs}(t), \\ 
\tilde U^{rs}_2(t, r) & \widetilde {r_M^R} ^-_{rs}(t)  <r< \widetilde {r_M^R} ^+_{rs}(t), \\ 
U_\gamma(r) & \widetilde {r_M^R} ^+_{rs}(t)< r< {r_M^R} ^-_R(t).
\end{cases}
\ee
Here, $U_{MM}^{rs}= U_{MM}^{rs} (r) $ is  a steady state with
\[
{U_M^{rs}}_M ( {r_M^L } ^+ _L(T_f)) \in  W_1^\rightarrow\big(\tildeU^{\alpha,\beta}_2 (T_f, {r_M^L } ^+ _L(T_f)) \big) \cap  W_2^\leftarrow\big(U_\gamma({r_M^R} ^-_R(T_f)) \big)
\]
and we recall  that $W_1^\rightarrow$ and $ W_2^\leftarrow$ are elementary waves given by \eqref{elementary}. The wave curves $ \widetilde { r_M^L } ^\pm _{rs},  \widetilde {r_M^R} ^\pm_{rs}(t)$ satisfy \eqref{wavecurves-left}, \eqref{wavecurves-right} with three states $\tildeU^{\alpha,\beta}_2, {U_M^{rs}}_M, U_\gamma$. The functions $\tilde U^{rs}_{1, 2} (t, r)$ are given by \eqref{form-of-U&r}, \eqref{BIC}, \eqref{if}.
 Denote by
\bel{T_rs}
T_{rs}^0= \sup \{t>T_f| \widetilde { r_M^L } ^- _{rs}(t)<  {r_M^L } ^-_L(t)\},
\ee
and we see immediately that \eqref{sol-Tsr} provides an exact solution for Problem $P-rs$ for all $0<t \leq T_{rs}^0$. Now for  $t>T_{rs}^0$, we give
\bel{sol-Tsr-1}
U^{rs}_M (t, r)=  \begin{cases}
U_L (t, r) &  \underline r <r<{r_M^L } ^+ _L(t), \\
U_M^{rs,1}& {r_M^L } ^+ _L(t)<r< \widetilde {r_M^R} ^-_{rs}(t),\\
U_M^{rs,0}(r)&  \widetilde {r_M^R} ^-_{rs}(t)  <r<{r_M^R} ^-_R(t), \\ 
U_\gamma(r) &{r_M^R} ^-_R(t)< r<\bar r,
\end{cases}
\ee
with $U_M^{rs,0}$ given by \eqref{sol-TsrM-0} and $U_M^{rs,1}$ the solution of the Riemann problem generated by initial data $U^{\alpha, \beta}_M$, $U_{MM} ^{rs,0}$ at the radius $r=  \widetilde { r_M^L } ^- _{rs}(T_{rs}^0)=  {r_M^L } ^-_L(T_{rs}^0)$ from the very moment $t= T_{rs}^0$.
Now we denote by
\bel{T_rs}
T_{rs}= \min \left(\sup \{t>T_{rs}^0| \overline {r_L^M} ^-_{rs}(t) >{r_L^M} ^+_L (t) \}, \sup \{t>T_{rs}^0|{r_M^R} ^-_R(t)> \widetilde {r_M^R} ^-_{rs}(t) \} \right),
\ee
where $ \overline {r_L^M} ^-_{rs}(t)$ is the lower bound of the 1-wave of the solution $U_M^{rs,1}=U_M^{rs,1}(t, r) $.  Together with \eqref{before-T_f}, \eqref{sol-Tsr}, \eqref{sol-Tsr-1}, we have a solution of Problem $P-rs$ for all $0<t<T_{rs}$.
\end{proof}
A similar analysis gives the result of Problem $P-rs$.

\begin{lemma}
\label{sr}
If  the first moment of interaction  $T_f< +\infty$ and the Euler model  \eqref{Euler1} admits  1-rarefaction of the left-hand problem and the 2-shock of the right-hand problem. That we have a solution of Problem $P-sr$ for all $0<t<T_{rs}$ for a given moment $T_{sr}$.
\end{lemma}

We now consider interactions after these moments $T_{ss}, T_{rs}, T_{sr}$. Indeed, following from the constructions in Lemmas~\ref{ss}, \ref{rs}, \ref{sr}, it is clear that possible interactions after these moments are also interplays of shock waves and rarefaction waves as is listed at the beginning of this section. Thus, for any fixed moment $t>0$, we have the solution of the triple Riemann problem. The estimation of the total variation given by \eqref{TV-triple} follows directly from Lemmas~\ref{lemma-E}, \ref{TV-R}.  We thus obtain the main conclusion of this section, that is, Theorem~\ref{Triple}.
%------------------------------------------

%=========================================================
\section{The initial value problem}
\label{sec:6}
\subsection{The Glimm method}
We construct an approximate solution of the Euler model \eqref{Euler1} with initial data
\bel{U0}
U (t, r)= U_0(r)= (\rho_0, v_0) (r), \quad r>0,
\ee
by using  a random choice method based or equivalently, the Glimm method on the generalized problem.
Let $\Delta r$ and $\Delta t$ denote the mesh lengths in space and in time  respectively, and let $(r_j , t_n)$ denotes the mesh point of the grid, where $r_j = j \Delta r$, $t_n = 0 + n \Delta t$. We assume the so-called \emph{CFL condition}:
\bel{CFL}
 {\Delta r \over  \Delta t }  > \max(|\lambda|, |\mu|),
\ee
insuring  that elementary waves other than those in the triple Riemann problem do not interact within one  time interval.

To construct the approximate solution $U_ {\Delta r }  = U_ {\Delta r}(t, r )$, we would first like to approximate the initial data  by a piecewise steady state solution of the Euler model given by \eqref{steady-Euler}. However, note that  some steady state solutions cannot be defined globally on $r>0$, we need more constructions.  Recall first that there exists four critical steady state solutions which pass the point $({m\over 2k ^2 }, \pm k ) $ denoted by $U_*^{P, \flat}, U_*^{P, \sharp}$, $U_*^{N, \flat}$, $U_*^{N, \sharp}$ according to \eqref{four-critical}. Another important remark is given in Theorem~\ref{steady-state}, that is, for given $r_0, U_0$, there exists always a steady solution $U=U(r)$ with $U(r_0)=U_0$ defined on $(0, r_0)$ if $r_0 <{ m \over 2k ^2 }$ or $(r_0, +\infty)$ if $r_0> { m \over 2k ^2 }$.
Now we denote by $U_{\Delta r, 0}^{j+1}= U_{\Delta r, 0}^{j+1}(r)= (\rho_{\Delta r, 0}^{j+1}, v_{\Delta r, 0}^{j+1})(r)$ the steady state solution of the Euler model  satisfying \eqref{steady-Euler} such that  $U_{\Delta r, 0}^{j+1}(r_{j+1})= U_0(r_{j+1})$ and we define:
\bel{rjs}
r_{j+1}^s : =\sup \{r>0|   v_{\Delta r, 0}^{j+1}(r)\neq  \pm k \} \chi_{\{r_{j+1}<{m \over 2k ^2}\}}(r)+ \inf \{r>0|   v_{\Delta r, 0}^{j+1}(r)\neq  \pm k \} \chi_{\{r_{j+1}>{m \over 2k ^2}\}}(r).
\ee
Note that if $r_{j+1}^s\neq 0 $ or  $r_{j+1}^s\neq +\infty$,  $r_{j+1}^s$  is the sonic point of the steady state $U_{\Delta r, 0}^{j+1}$.
We now denote by $U_{0, *}^{j+1}= (\rho_{0, *}^{j+1}, v_{0, *}^{j+1})$ the unique critical steady state solution satisfying
\bel{same-steady}
\sgn(v_{0, *}^{j+1})= \sgn(v_{\Delta r, 0}^{j+1}), \quad \sgn(|v_{0, *}^{j+1}|-k )= \sgn(|v_{\Delta r, 0}^{j+1}|-k ).
\ee
On the interval $(r_j, r_{j+2})$, we have the following possible constructions.
\bei
\item If $U_{\Delta r, 0}^{j+1}$ is well-defined on $(r_j, r_{j+2})$, we approximate the initial data $U_0$ by $U_{\Delta r, 0}^{j+1}$  on the interval.
\item If $U_{\Delta r, 0}^{j+1}$ vanishes at $r_{j+1}^s$ and $r_{j+1}< {m \over 2k ^2}$, then we approximate the initial data on $(r_{j+1}^s, r_{j+2})$ by
\bei
\item $U_{\Delta r, 0}^{j+3}$ if $r_{j+3}^s \notin (r_{j+1}^s, r_{j+2})$;
\item $U_{0, *}^{j+1}$  if $r_{j+3}^s \in (r_{j+1}^s, r_{j+2})$ for $U_{0, *}^{j+1}$ given by \eqref{same-steady}. Note that this case happens at most once if $r_{j+1}<{m \over 2k ^2}  <r_{j+3}$ and $r_{j+3}^s > r_{j+2}$.
\eei
\item If $U_{\Delta r, 0}^{j+1}$ vanishes at $r_{j+1}^s$ and $r_{j+1}> {m \over 2k ^2}$, then we approximate the initial data on $(r_j , r_{j+1}^s )$ by
\bei
\item $U_{\Delta r, 0}^{j-1}$ if $r_{j-1}^s \notin ( r_j, r_{j+1}^s)$;
\item $U_{0, *}^{j+1}$  if $r_{j-1}^s \in( r_j, r_{j+1}^s)$. Also,  this case happens at most one time if  $r_{j-1}<{m \over 2k ^2}  <r_{j+1}$ and $r_{j-1}^s< r_j$.
\eei
\eei
Following the ideas above, we can now approximate the initial data on $(r_j , r_{j+2})$ for $j$ even:
\bel{initial-state} 
U_{\Delta r,0}(r ) = \begin{cases}
\bar U_{\Delta r, 0}^{j+1- 2\sgn(r_{j+1}- {m \over 2k ^2}  )} (r)&  r_j <r <\mathcal M(r_j , r_{j+1}^s),\\
U_{\Delta r, 0}^{j+1}(r)& \mathcal M(r_j , r_{j+1}^s)< r< \mathcal M(r_{j+1}^s, r_{j+2}), \\
\bar U_{\Delta r, 0}^{j+1- 2 \sgn(r_{j+1}- {m \over 2k ^2}  )} (r) &  \mathcal M(r_{j+1}^s, r_{j+2})< r< r_{j +2 },
\end{cases}
\ee
where we give the operator $\mathcal M $ by
\bel{Op-M}
\mathcal M(x, y)=\begin{cases}
 \min(x, y) & r< {m \over 2k ^2 }, \\
 \max(x,y)  & r> {m \over 2k ^2 },
 \end{cases}
\ee
and
\[
\bar U_{\Delta r, 0}^{j+1-2 \sgn(r_{j+1}- {m \over 2k ^2}  )}(r)= \begin{cases}
U_{\Delta r, 0}^{j+1-2 \sgn(r_{j+1}- {m \over 2k ^2}  )}(r) & r_{ j+1-2 \sgn(r_{j+1}- {m \over 2k ^2})} ^s \notin ( r_j, r_{j+1}^s) \cup ( r_{j+1}^s , r_{j+2}),  \\
 U_{0, *}^{j+1} (r)& \text{else},
\end{cases}
\]
with the sonic point
$r_{j+1}^s$ given by \eqref{rjs} and the critical steady state solution $ U_{0, *}^{j+1} $ satisfying \eqref{same-steady}.
Assume now that the approximate solution has been defined for $t_{n-1}\leq t < t_{n}$. To complete the definition of $U_{\Delta r}$,  it suffices to define the  solution on $t_n \leq t < t_{n+1}$. Let $\theta_n$ be a given  equidstributed sequence on the interval $(-1, 1)$ and introduce the point related to the randomly choose values:
\bel{equi}
r_{n, j+1 } : = (\theta_n + j)\Delta r,\qquad  j>0.
\ee
Following the idea before, we denote by $U_{\Delta r, n}^{j+1}= U_{\Delta r, n}^{j+1}(r)$ the steady state solutions passing the point $(r_{n, j+1 }, U_{\Delta r} (nt-, r_{n, j+1}))$ and the sonic point
\[
r_{n,j+1}^s : =\sup \{r>0|   v_{\Delta r,n }^{j+1}(r)\neq  \pm k \} \chi_{\{r_{n,j+1}<{m \over 2k ^2}\}}(r)+ \inf \{r>0|   v_{\Delta r, n }^{j+1}(r)\neq  \pm k \} \chi_{\{r_{n, j+1}>{m \over 2k ^2}\}}(r),
\]
together the critical steady state solution $U_{n , *}^{j+1}= (\rho_{n, *}^{j+1}, v_{n, *}^{j+1}) $ such that
\bel{U*n}
\sgn(v_{n, *}^{j+1})= \sgn(v_{\Delta r, n}^{j+1}), \quad \sgn(|v_{n, *}^{j+1}|-k )= \sgn(|v_{\Delta r, n }^{j+1}|-k ).
\ee

Now suppose that $U_{\Delta r}$ is constructed for all $t< t_n$. The construction of the approximate solution on the time interval $t_n \leq t < t_{n+1}$ is similar to the approximation of the initial data:
\bei
\item {\bf The steady state solution step. }  On the level $t=t_n $, on  the interval $(r_j , r_{j+2})$ with $n+j$ even
 \bel{piecewise-steady-state}
 U_{\Delta r,n }(r ) = \begin{cases}
\bar U_{\Delta r, n}^{j+1-2 \sgn(r_{n,j+1}- {m \over 2k ^2}  )} (r)&  r_j \leq r < \mathcal M(r_j , r_{n,j+1}^s),\\
U_{\Delta r, n}^{j+1}(r)& \mathcal M(r_j , r_{n,j+1}^s)< r< \mathcal M(r_{n,j+1}^s, r_{j+2}), \\ 
\bar U_{\Delta r, n }^{j+1- 2 \sgn(r_{j+1}- {m \over 2k ^2}  )} (r) & \mathcal M(r_{n,j+1}^s, r_{j+2})< r< r_{j +2 },
\end{cases}
\ee
where $\mathcal M(\cdot, \cdot)$ is the operator  given by \eqref{Op-M} and
\[
\bar U_{\Delta r, n}^{j+1-2 \sgn(r_{n,j+1}- {m \over 2k ^2}  )}(r)= \begin{cases}
U_{\Delta r, n}^{j+1- 2 \sgn(r_{n,j+1}- {m \over 2k ^2}  )}(r) & r_{ j+1-2 \sgn(r_{n,j+1}- {m \over 2k ^2})} ^s \notin ( r_j, r_{n,j+1}^s),  \cup ( r_{n,j+1}^s , r_{j+2}),  \\
 U_{n, *}^{j+1} (r) & \text{else},
\end{cases}
\]
with $ U_{n, *}  ^{j+1} $ given by \eqref{U*n}.
It is direct to observe that if a steady state solution reaches its sonic point in a cell, then the nearest discontinuity is replaced by this sonic point, then this construction guarantees that there exists at most one point of discontinuity in $(r_{j-1}, r_{j+1})$, $j+n $ even.
\item  {\bf The generalized  Riemann problem  step.} Denote by $r^d_j $ the point of discontinuity in $r_{j-1}<r<  r_{j+1}$ and we then define the approximate solution $U_{\Delta r}$ on the rectangle $ \{t_n < t < t_{n+1}, r_{j-1}<r<  r_{j+1}\}$, $n+j$ even:
 \bel{generalized Riemann problem}
 U_{\Delta r}(t, r ) = \begin{cases}
  U_{\mathcal R} ^{(j-1, j+1)}(t, r), &  \text{$r_j^d - r_{j-2}^d = 2 \Delta r$ and $r_{j+2}^d - r_j^d = 2 \Delta r$ },  \\
  U_{\mathcal {TR} }^{(j-3, j+1)}(t, r), &   r_j^d - r_{j-2}^d < 2 \Delta r,\\ 
  U_{\mathcal {TR}}^{(j-1, j+3)}(t, r), &   r_{j+2}^d - r_j^d < 2 \Delta r,
\end{cases}
\ee
where $U_{\mathcal R}^{(j-1, j+1)}$ is the solution of the generalized Riemann problem at the time level $t = t_n$ on $(r_{j-1}, r_{j+1})$ with two steady states separated by a discontinuity at $r_j^d$ and $U_{\mathcal {TR} }^{(j-3, j+1)}$ the solution of the triple Riemann problem at the time level $t = t_n$ on the interval  $(r_{j-3}, r_{j+1})$  with the three steady states separated  by discontinuities  at $r_{j-2}^d, r_j ^d$.
\eei
This completes the construction of the  approximate solution $U_{\Delta r}= U_{\Delta r}(t, r)$ on $[0, +\infty) \times (0, +\infty)$ by the Glimm scheme.

%---------------------------------------------------------------------------------
\subsection{Existence of Cauchy problem}
The Glimm scheme provides as an approximate solution which indeed converges to an exact weak solution.

\begin{theorem}[Global existence theory]
\label{global-existence}
Consider the Euler model with source term describing fluid flows \eqref{Euler1}. For  any given initial density $\rho_0=\rho_0(r) > 0$ and velocity  $v_0$  such that
\[
TV \big(\ln \rho_0 \big) + TV  ( v_0) < + \infty,
\]
and any given time interval (possibly infinite) $(0, T) \subset (0, +\infty)$,  there exists a weak solution $\rho=\rho(t,r),v=v(t,r)$ defined on $(0, T)$ such that the initial condition holds in the sense that  $\rho(0, \cdot)= \rho_0,v(0, \cdot)= v_0)$ and for any fixed moment $T'\in (0, T)$
\[
\sup_{t \in [0, T']} \Big(
TV \big( \ln \rho(t,\cdot) \big)
+ TV(v)\big) < +\infty.
\]
\end{theorem}
To prove Theorem~\ref{global-existence}, we first need an estimation of the total variation.  See the following lemma.
\begin{lemma}
\label{lemmaTV}
Let $U_{\Delta r}= (\rho_{\Delta r}, v_{\Delta r} )$ be the approximate solution of the Euler model \eqref{Euler1} constructed by the Glimm method, then for any two neighboring time  interval $t_n , t_{n+1}$, we have a constant $C>0$ such that
\[
TV \big(\ln \rho_ {\Delta r}(t_{n+1} + , \cdot)\big) - TV \big(\ln \rho_{\Delta r} (t_n+  , \cdot)\big) \leq C \Delta t.
\]
\end{lemma}
From Lemma~\ref{lemmaTV}, we have, for any given $0<t<+\infty$,
\bel{estimation-TV}
TV \big(\ln \rho_{\Delta r} (t , \cdot)\big) \leq  TV \big(\ln \rho_ {\Delta r} (0  , \cdot)\big) e^ {C_1 t},
\ee
where $C_1$ is a constant.
\begin{proof}
On the time level $t=t_{n+1}$, we consider the interval $ (r_{n+1, j-1}, r _{n+1, j+1} )$ with $n+j$ even. According to \eqref{equi}, $r_{n\pm 1, j+1}$ is the point determined by a chosen random value. Following from the construction of the Glimm method, $ (r_{n+1, j-1}, r _{n+1, j+1} )$ only contains one point of discontinuity  which we write as $r_{j, n } ^d$. According to 	Lemma~\ref{TV-R}, we have
\[
TV \big(\ln \rho (t_{n+1} + , \cdot)\big)= \sum_ j  | \ln \rho (t_{n+1} + , r _ {j+1, n}^d ) - \ln \rho (t_{n+1} + , r _ {j, n}^d ) | \big(1+ C (\Delta t)\big ).
\]
Now we notice that  there are portions of three possible waves generated by either the generalized Riemann problem or the triple Riemann problem lying in the interval $(r_{n+1, j-1}, r _{n+1, j+1} )$.  We write these waves as $\omega_{l, m, r}$ from left to right, staring from points of  discontinuity (reading $r_{l,n} ^d,  r_{m,n}^d, r_{r,n}^d$  respectively) in $(r_{j-2}, r_j], (r_j, r_{j+2}], (r_{j+2}, r_{j+4}] $ at the  time level $t= t_ n$ respectively.

We observe that the wave $\omega_l $ is either a zero strength wave or a in  $ (r_{n+1, j-1}, r _{n+1, j+1} )$ or  the wave generated by a steady state $U_L$ such that $U_L (r_{n+1, j-1})= U_{\Delta r }  (t_{n+1}-, r_{n +1, j-1})$  and another steady state $U_M$ such that $U_M=U_{\Delta r}  (t_{n +1}-, r_{n+1, j+1}) $ depending on if the position of the  the randomly chosen  point $r_{n+1, j-1}$ is  closer to $r_{n,j}^d$ or closer to $\omega_1$) .  Similarly, $\omega_3$ is either a zero strength wave in $ (r_{n+1, j-1}, r _{n+1, j+1} )$ or a wave given by the state $U_M$ such that $U_M( r_{n+1, j+1})=U_{\Delta r } (t_{n+1}-, r_{n+1, j+1}) $ and another state $U_R$ such that $U_R=U_{\Delta r }  (t_{n+1}-, r_{n+1, j+3})$. Turning to the wave $\omega_2$, it is generated by $U_L$ and $U_M$ or  $U_M$ or $U_R$.  According to  to Lemma~\ref{lemma-E}, we have the
result by adding $j$ on the time level $t=t_n$.

\end{proof}

Now since the uniform BV bound on a given time interval $(0, T)$ (established below) is known, Helly's theorem gives immediately the fact that there exists a subsequence of  $\Delta r \to 0$ such that we have  a limit function $U= U(t, r)$ and $U_{\Delta r}  (t,r) \to U (t, r)$ pointwise a.e. and in $L_{loc}^1$ at each fixed time $t$.  Moreover, the limit function $U= U (t, r)$ is a weak solution of the Euler model \eqref{Euler1}, \eqref{U0}. This ends the proof of Theorem~\ref{global-existence}.
 
 %================================================================
 
 \section{Conclusion}
 In the article, we considered a kind a Euler equation with a particular source term depending on the sound speed and the body mass. We first presented the hyperbolicity and the nonlinear genuinity of the equation.  We gave then an analysis of the steady state solutions of this model and give a classification of these steady states with respect to the behaviour of the sonic points. We then considered the generalized Riemann problem whose initial data are two constant steady state solutions and proved their existence by giving an analytical formula of the solution. We also proved the existence for  a so-called triple Riemann problem with three different steady state solutions. We were then able to use the Glimm method to construct a sequence of the solutions to the initial value problem of the  Euler equation and prove it existence with a control of the total variation.

 %=================================

%=============================================================================


\begin{thebibliography}{10}


 \bibitem{DL} \auth{Dafermos C.M. and Hsion L.,} Hyperbolic systems of balance laws with inhomogeneity and dissipation, Indiana Univ. Math. J. 31 (1982), 471--491.
 
 \bibitem{GLIMM}  \auth{Glimm J.,} Solutions in the large for nonlinear hyperbolic systems of equations, Comm.PureAppl. Math. 18 (1965), 697--715.
 
 \bibitem{GL} \auth{GRUBICN. and LEFLOCHP.G.,} Weakly regular Einstein-Euler spacetimes with Gowdy symmetry. The global areal foliation, Arch. Rational Mech. Anal. 208 (2013), 391--428.
 
\bibitem{HL} \auth{Hong J.M. and LeFloch P.G.,}
A version of the Glimm method based on generalized Riemann problems,
Port. Math. 64 (2010), 635--650.


\bibitem{Lax57} \auth{Lax P.D.,}
Hyperbolic systems of conservation laws. II,
Comm. Pure Appl. Math. 10 (1957), 537--566.

\bibitem{PLF} \auth{LeFloch P.G.}
Hyperbolic conservation laws on spacetimes, in "nonlinear conservation laws and applications",
IMA Vol. Math. Appl. 153; Springer, New York, 2011.

\bibitem{LeFloch-book} \auth{LeFloch P.G.,}
{\em Hyperbolic systems of conservation laws},
Lectures in Mathematics, ETH Z\"urich, Birkh\"auser, 2002.

\bibitem{PLF-SX-one} \auth{LeFloch P.G. and Xiang S.,}
Weakly regular fluid flows with bounded variation on the domain of outer communication of a Schwarzschild black hole spacetime, 
J. Math. Pure Appl. 106 (2016), 1038--1090.

\bibitem{PLF-SX-two} \auth{LeFloch P.G. and Xiang S.,}
Weakly regular fluid flows with bounded variation on the domain of outer communication of a Schwarzschild black hole spacetime. II, 
J. Math. Pure Appl. 122 (2019), 272--317.


\bibitem{Liu} \auth{Liu T.P.,}
The deterministic version of the Glimm scheme,
Comm. Math. Phys. 57 (1977), 135--148.

\bibitem{Liu2} \auth{Liu T.P.,}
Quasilinear hyperbolic systems,
Comm. Math. Phys. 68 (1979), 141--172.

\bibitem{Nishida} \auth{NISHIDA T.,}  Global solutions for an initial boundary value problem of a quasilinear hyperbolic system, Japan Acad. 44 (1964), 642?646.

\bibitem{SJ} \auth{Smoller J.}
Shock waves and reaction-diffusion equations,
Springer-Verlag, Berlin/ New York, 1983.
\end{thebibliography}
\end{document}